\documentclass[11pt]{article}
\usepackage[fleqn]{amsmath}
\usepackage{amssymb, amsthm, latexsym, amsfonts, epsfig,  graphicx,epstopdf,subcaption,caption,mathrsfs}
\usepackage{float}
\usepackage{ulem}
\usepackage{enumitem}
 \usepackage{color}
 \usepackage{appendix}
\numberwithin{equation}{section}%分节对方程计数
\DeclareMathOperator*{\esssup}{ess\,sup}
\def\R{\mathbb{R}}

\def\kasten{$~~\mbox{\hfil\vrule height6pt width5pt depth-1pt}$ }

\usepackage{geometry}
\newtheorem{Theorem}{Theorem}[section]
\newtheorem{Corollary}[Theorem]{Corollary}
\newtheorem{Definition}[Theorem]{Definition}
\newtheorem{Proposition}[Theorem]{Proposition}
\newtheorem{Lemma}[Theorem]{Lemma}

\newtheorem{Remark}[Theorem]{Remark}

\begin{document}

\def\kasten{$~~\mbox{\hfil\vrule height6pt width5pt depth-1pt}$}
\makeatletter\def\theequation{\arabic{section}.\arabic{equation}}
\makeatother

\newtheorem{theorem}{Theorem}
\centerline{\large \bf The random dynamical pitchfork bifurcation with additive L\'evy noises
\footnote{This work was partly supported by the
 NSFC grant XXXX.}\hspace{2mm}
\vspace{1mm}\vspace{1mm}\\ }
\smallskip
\centerline{\bf Ziying He$^{a}\footnote{ziyinghe@whut.edu.cn}$,
\bf Xianming Liu$^{*b}\footnote{xmliu@hust.edu.cn}$}
\smallskip
\centerline{${}^a$ School of Science, Wuhan University of Technology,}
\centerline{Wuhan, Hubei 430070, China,}
\centerline{${}^b$ School of Mathematics and Statistics, Huazhong University of Sciences and Technology,}
\centerline{Wuhan, Hubei 430074, China}

\begin{abstract}
This paper concerns the effects of additive non-Gaussian L\'evy noises on the pitchfork bifurcation. We consider two types of noises, $\alpha$-stable process and the truncated process. Under both $\alpha$-stable process and the truncated process, the classical pitchfork bifurcation model exists a unique invariant measure. The Lyapunov exponent associated with the invariant measure is always negative for the system under the truncated case. While the stochastic pitchfork bifurcation still occurs. In both cases, the attractivity uniformity, the finite-time Lyapunov exponent, and the dichotomy spectrum behave varies with the bifurcation parameter changing.

Compared with Brownian motion, there is two key difficulties for the L\'evy processes. The stationary density can not be solved explicitly, thus we have to estimate it properly. This is overcome by the strong maximum principle. The bilateral suprema for L\'evy processes need to be analyzed. This is acquired by the strong Markov property. Based on them, we establish the main results indicating stochastic pitchfork bifurcation.

\end{abstract}


\section{Introduction}
Stochastic bifurcations are widespread in nature. Stochastic bifurcation theory investigates the qualitative changes in parametrized families of Random dynamical system. As the meeting of two mathematical cultures, stochastic dynamics can be addressed through the approaches of stochastic analysis and dynamical systems. Being the classical method of the former, Markov semigroup leads to the notion of P-bifurcation by investigating a qualitative change in the shape of it's stationary measure. The latter derives the notion of D-bifurcation by exploring the branching of it's invariant measures and the stability. The two concepts complement each other \cite{Ludwig Arnold}. We concern D-bifurcation primarily.

Stochastic factors can be depicted with Gaussian or non-Gaussian processes in line with different versions of the central limit theorem. Influenced by stochastic factors, dynamical systems  under L\'evy processes deserve investigated. We are interested in the effects of L\'evy processes on pitchfork bifurcation. The buckling beam example offers an instance of pitchfork bifurcation. Consider the following stochastic dynamical system
$$\mathrm{d}X_t=(\beta X_t-X_t^3)\mathrm{d}t+\sigma\mathrm{d}L_t^{\alpha}, \quad X_0\in\mathbb{R},$$
with the rotational invariant L\'evy process $(L_{t}^{\alpha})_{t\in\R}$ defined on a probability space $(\Omega,\mathcal{F},\mathbb{P})$. 

In the case of deterministic with $\sigma=0$, the system displays a pitchfork bifurcation on equilibria. If $\beta<0$, there exists a unique equilibrium $0$ which is attractive globally. If $\beta>0$, there exist three equilibria $0$, and $\pm\beta$ with the first repulsive and the additional attractive.

In the case of Gaussian noise with $\alpha=2$, the system was surveyed by H. Crauel and F. Flandoli in 1998 \cite{Hans Crauel and Franco Flandoli}. In the classical frame of stochastic bifurcation theory established by L. Arnold and co-workers \cite{Ludwig Arnold}, the system exhibits that the additive noise destroys a pitchfork bifurcation. Specifically, the system has a unique invariant measure with a negative Lyapunov exponent. Right after that, M. Callaway, T. Doan, J. Lamb and M. Rasmussen argued that the pitchfork bifurcation was not destroyed by additive noise in 2017 \cite{Mark Callaway}. They discover additional indicators, including attractive uniformity, finite-time Lyapunov exponent and dichotomy spectrum, to detect stochastic bifurcation. They also propose a concept of uniform topological equivalence to characterize stochastic bifurcation.

In the case of non-Gaussian noises defined in \eqref{levy-process} with L\'evy measure being \eqref{nu}, we establish a sequence of results detecting the stochastic pitchfork bifurcation of the system. We call the l\'evy process defined in \eqref{levy-process} as non-truncated if $\nu(B_1^c)\neq0$, and truncated if $\nu(B_1^c)=0$. 

The system exists a unique invariant measure under the random dynamical system. This is proved by the relation between the invariant measure under Markov semigroup and that under random dynamical system. Furthermore, the Lyapunov exponent associated with the invariant measure exists and is negative in the truncated case \eqref{levy-process}. This is not trivial compared to that of Brownian motion, since that the stationary density is not solvable explicitly in the non-Gaussian case. The existence is confirmed by the finite second moment and multiplicative ergodic theorem. The estimation towards the stationary density is by the strong maximum principle. Then the negativity of the Lyapunov exponent is estimated properly. 
In the frame of the stochastic bifurcation theory established 
by L. Arnold, the pitchfork bifurcation is destroyed by the non-Gaussian noise also. 

Furthermore, the attractive uniformity of the random attractor, the sign of the finite-time Lyapunov exponent and the dichotomy spectrum perform differently with different value of $\beta$. The position of the singleton random attractor is estimated on basis of the support of stationary density being the whole line. The L\'evy processes stay in the vicinity of original for very long time with a positive probability is certified by the strong Markov property. Based on these facts, the attractive uniformity of the random attractor, the sign of the finite-time Lyapunov exponent and the dichotomy spectrum are analyzed. The behaviors of these objects, proposed by J. Lamb and co-workers \cite{Mark Callaway},  indicate the occurring of the stochastic pitchfork bifurcation. 

The present paper is arranged as follows. In section \ref{sectionSetup}, we recall the related concepts about invariant measure under the random dynamical system and that under Markov semigroup. In section \ref{sectionmain}, we describe the main results of this paper. In section \ref{sectionbasic}, we conduct some basic analyses containing exponential ergodicity of the Markov semigroup, the collapse of the random attractor and the finite moment of the invariant measure under the Markov semigroup. In section \ref{sectionkey}, we verify the key Lemmas required for the main results. In section \ref{Proof of the main results}, we prove the main results.

\section{Setup}\label{sectionSetup}
Let $(\Omega,\mathcal{F},\mathbb{P})$ be a probability space and $(X,\mathcal{B})$ be a metric space. The probability space  $(\Omega,\mathcal{F},\mathbb{P})$ is equipped with the filtration $\mathcal{F}_t$. We denote the past filtration and the future filtration as $\mathcal{F}_t^-$ and $\mathcal{F}_t^+$ respectively \cite{David Applebaum, Ludwig Arnold}. Consider the following stochastic differential equations (SDEs):
\begin{eqnarray}\label{d-dim-sys}
\mathrm{d}X_{t}=b(X_{t})\mathrm{d}t+\sigma\mathrm{d}L_{t},\quad  X_0\in \mathbb{R}^{d},
\end{eqnarray}
where $\{L_{t}\}_{t\in\R}$ is the rotational invariant L\'evy process. It's L\'evy-It$\hat{o}$ decomposition is
\begin{eqnarray}\label{levy-process}
L_t=\left\{\begin{array}{ll}                                  
                                        \int_{B_1}z\widetilde{N}(t,\mathrm{d}z)+\int_{B_1^c}z N(t,\mathrm{d}z), & \mbox{ if} \quad\nu(B_1^c)\neq0,\\
                                         \int_{B_1}z\widetilde{N}(t,\mathrm{d}z), & \mbox{ if}\quad\nu(B_1^c)=0,
                \end{array}\right.
\end{eqnarray}
where $B_r$ is the ball with the origin as center and $r(>0)$ as radius.  
For every $t\in [0,\infty), A\in\mathcal{B}(\mathbb{R}^d)$, the Poisson random measure is
$$N(t,A)(\omega)=\#\{0\leq s\leq t;\Delta L_s(\omega)\in A\}.$$
Further, the compensated Poisson random measure is
$$\widetilde{N}(t,A)=N(t,A)-t\nu(A),$$
with L\'evy measure $\nu$, that is it satisfies
$$\int_{\mathbb{R}^d\setminus\{0\}}(|z|^2\wedge1)\nu(\mathrm{d}z)<+\infty.$$
The corresponding characteristic function is
\begin{eqnarray}\label{characteristic-function}
\varphi_t(u)=\left\{\begin{array}{ll}\exp\bigg\{t\int_{\mathbb{R}^d\setminus\{0\}}[e^{i(u,z)}-1-i(u,z)\chi_B(z)]\mu(\mathrm{d}z)\bigg\}, & \mbox{ if} \quad\nu(B_1^c)\neq0,\\
\exp\bigg\{t\int_{B_1\setminus\{0\}}[e^{i(u,z)}-1-i(u,z)]\mu(\mathrm{d}z)\bigg\}, & \mbox{ if}\quad\nu(B_1^c)=0.
                             \end{array}\right.
\end{eqnarray}
The Markov semigroup $(P_t)_{t\geq0}$ associated with system \eqref{d-dim-sys} is
$$P_{t}f(x)=\mathbb{E}^xf(X_t),\quad t\geq0,\quad x\in \mathbb{R}^d,$$
for any $f\in\mathcal{B}_b(\mathbb{R}^{d})$, where $\mathcal{B}_b(\mathbb{R}^{d})$ is the space of all bounded Borel measurable functions on $\mathbb{R}^{d}$ equipped with the sup-norm. 
\begin{Definition}{\bf (Invariant measure under semigroup)}
Let $(X,\mathcal{B})$ be a measurable space and let $\mathcal{B}_b(X)$ be the space of all bounded $\mathcal{B}$-measurable functions on $X$ equipped with the sup-norm. If $\{P_t\}_{t\geq0}$ is a semigroup of bounded linear operators on the space $\mathcal{B}_b(X)$, then a bounded measure $\mu$ on $\mathcal{B}$ is called invariant for $\{P_t\}_{t\geq0}$ if 
$$\int_{X}P_tf\mathrm{d}\mu=\int_{X}f\mathrm{d}\mu,  \quad\forall f\in \mathcal{B}_b(X).$$
\end{Definition} 
\noindent Define 
$$D_{\mathcal{L}_{b}}=\{f\in C^1(\mathbb{R}^{d}): \exists g_f\in C^1(\mathbb{R}^{d}) \mbox{ such that } \, \lim_{t\downarrow0}||\frac{P_tf-f}{t}-g_f||=0 \}.$$
\noindent For $f\in D_{\mathcal{L}_{b}}$, the generator of the Markov semigroup associated with system \eqref{d-dim-sys} is
$$\mathcal{L}_{b}f(x)=b(x)\cdot\bigtriangledown f(x)+\mathcal{L}f(x),$$
where
\begin{eqnarray}\label{generator}
\mathcal{L}f(x)=\left\{\begin{array}{ll}
                                      P.V.\int_{\mathbb{R}^{d}}[f(x+z)-f(x)]\nu(\mathrm{d}z), & \mbox{ if} \quad\nu(B_1^c)\neq0,\\
                                      \\
                                      P.V.\int_{B_1}[f(x+z)-f(x)]\nu(\mathrm{d}z), & \mbox{ if} \quad\nu(B_1^c)=0.
                                      \end{array}
                                      \right.
\end{eqnarray}
For the refinement definition of $D_{\mathcal{L}_{b}}$ refer \cite{Claudia Bucur, Qing Han}. The density function $p(t,x)$ of the process $X_{t}$ in system \eqref{d-dim-sys} satisfies the Fokker-Planck equation, i.e.
\begin{eqnarray}\label{fokker-planck}
p_t=\mathcal{L}_{b}^{*}p,\,\,p\geq0,\,\,\int_{\mathbb{R}^{d}}p(t,x)\mathrm{d}x=1,
\end{eqnarray}
where $\mathcal{L}_{b}^{*}$ is the adjoint operator of the generator $\mathcal{L}_{b}$ for system \eqref{d-dim-sys}. 
Specifically, take 
\begin{eqnarray}\label{nu}
\nu(\mathrm{d}z)=\nu_\alpha(\mathrm{d}z)=\frac{c_\alpha}{|z|^{d+\alpha}}\mathrm{d}z,
\end{eqnarray}
we arrive at the $\alpha$-stable L\'evy motion $L_{t}^\alpha$,
where the constant $c_\alpha=\frac{\alpha}{2^{1-\alpha}\pi^{d/2}}\frac{\Gamma((d+\alpha)/2)}{\Gamma(1-\alpha/2)}$, and the Gamma function $\Gamma(\lambda):=\int_{0}^{\infty}t^{\lambda-1}e^{-t}\mathrm{d}t$ for every $\lambda>0$ \cite{JinqiaoDuan}. In this case $\mathcal{L}f(x)=-(-\Delta)^{\alpha/2}$ \cite{Mateusz Kwasnicki} and the corresponding operator in Fokker-Planck equation \eqref{fokker-planck} associated with system \eqref{d-dim-sys} is 
$$\mathcal{L}_b^*p=-(-\Delta)^{\alpha/2}p-\nabla(b\cdot p).$$

We consider the stochastic bifurcation phenomenon under the frame of random dynamical system (RDS). For this, recall the related definitions.
\begin{Definition}{\bf(Metric dynamical system)}\cite{Ludwig Arnold}
A measurable dynamical system $(\theta(t))_{t\in\mathbb{T}}$ on a probability space $(\Omega,\mathcal{F},\mathbb{P})$ for which each $\theta(t)$ is an endomorphism is called a measure preserving or metric dynamical systems and is denoted by $\Sigma=(\Omega,\mathcal{F},\mathbb{P},(\theta(t)))_{t\in\mathbb{T}})$, or for short, by $\theta(\cdot)$ or $\theta$.
\end{Definition}

\begin{Definition}{\bf(Random dynamical system)}\cite{Ludwig Arnold}
A measurable random dynamical system (RDS) on the measurable space $(X,\mathcal{B})$ over a metric dynamical system $(\Omega,\mathcal{F},\mathbb{P},(\theta(t))_{t\in \mathbb{T}})$ with time $\mathbb{T}$ is a mapping
$$\varphi:\mathbb{T}\times\Omega\times X\to X,\qquad(t,\omega,x)\mapsto\varphi(t,\omega,x),$$
with the following properties:
\flushleft
\begin{enumerate}
\item[(i)] Measurability: $\varphi$ is $\mathcal{B}(\mathbb{T})\otimes\mathcal{F}\otimes \mathcal{B}$, $\mathcal{B-}$measurable.
\item[(ii)] Cocycle property: The mappings $\varphi(t,\omega):=\varphi(t,\omega,\cdot):X\to X$ form a cocycle over $\theta(\cdot)$, i.e. they satisfy
$$\varphi(0,\omega)=id_{X}\qquad\mbox{for all } \omega\in\Omega \qquad(\mbox{if } \,\, 0\in\mathbb{T}).$$
$$\varphi(t+s,\omega)=\varphi(t,\theta(s)\omega)\circ\varphi(s,\omega),\qquad\mbox{for all }\,\, s,t\in\mathbb{T},\,\, \omega\in\Omega.$$
\end{enumerate}
\end{Definition}

\begin{Definition}{\bf (Invariant measure under random dynamical system)}\label{inv-measure}\cite{Ludwig Arnold}
Given a measurable RDS $\varphi$ over $\theta$, a probability measure $\mu$ on $(\Omega\times X,\mathcal{F}\otimes\mathcal{B})$ is said to be an invariant measure for the RDS $\varphi$, or $\varphi-$invariant, if it satisfies\\
1. $\Theta(t)\mu=\mu$ for all $t\in\mathbb{T}$,\\
2. $\pi_{\Omega}\mu=\mathbb{P}$.
\end{Definition}
\noindent Define
$$\mathcal{P}_P(\Omega\times X):=\{\mu \mbox{ probability on } (\Omega\times X, \mathcal{F}\otimes\mathcal{B})\mbox{ with marginal } \mathbb{P} \mbox{ on } (\Omega,\mathcal{F})\}$$
\begin{Definition}{\bf (Markov measure)}\label{Mark-measure}
\cite{Ludwig Arnold}
Let $\varphi$ be a measurable RDS with two-sided time with past $\mathcal{F}^-$ and future $\mathcal{F}^+$. A probability measure $\mu\in\mathcal{P}_P(\Omega\times X)$ for which the factorization $\omega\mapsto\mu_{\omega}$ is $\mathcal{F}^-$-measurable or $\mathcal{F}^+$-measurable, i.e. $\mathbb{E}(\mu.|\mathcal{F}^\pm)=\mu.$ $\mathbb{P}$-a.s., is called a Markov measure. More specifically, an $\mathcal{F}^-/\mathcal{F}^+$-measurable $\mu$ is called a forward/backward Markov measure.
\end{Definition}
\begin{Remark}\label{relation invariant measure}
The invariant measure under semigroup and that under random dynamical system are not the same objects. They may behave very differently. Consider the one dimensional Ornstein-Uhlenbeck process \cite{Hans Crauel Markov}
$$\mathrm{d}X_t=\beta X_t\mathrm{d}t+\sigma\mathrm{d}W_t,\qquad  X_0\in\mathbb{R},$$
with $\beta\neq0$ and $\sigma\neq0$ for two sided time $\mathbb{T}=\mathbb{R}$. 

For all $\beta\in\mathbb{R}$, there is a unique invariant measure under the random dynamical system $\mu$ with
$\mu_\omega=\delta_{X_t(\omega)}$, where
$$X_t(\omega)=\left\{\begin{array}{ll} 
                 -\int_0^\infty e^{-t\beta}\sigma\mathrm{d}W_t,\quad \mbox{if }\beta>0,\\
                       \\
                        \int_{-\infty}^0 e^{-t\beta}\sigma\mathrm{d}W_t,\quad \mbox{if }\beta<0.
\end{array}\right.$$

For $\beta<0$, the invariant measure $\rho$ under the associated Markov semigroup is the density of the Gaussian distribution $\mathcal{N}(0,-\frac{\sigma^2}{2\beta})$. For $\beta>0$, there exists no invariant measure under the Markov semigroup.  

While the invariant measure under semigroup and that under random dynamical system are are connected through the Markov measure. If $\beta<0$, then $\mu$ is Markov. If $\beta>0$, then $\mu$ is not Markov.
There is a one-to-one correspondence between the invariant Markov measure under random dynamical system and the invariant measure under semigroup \cite{Hans Crauel and Franco Flandoli, Hans Crauel Markov}. They are connected through the following relations,
$$\rho=\mathbb{E}\mu, \quad \mbox{ i.e. } \rho(B)=\int_\Omega\mu_\omega(B)\mathrm{d}P(\omega),$$
and
$$\mu_\omega=\lim_{t\to\infty}\varphi(t,\theta_{-t}\omega)\rho, \quad a.s.$$
where $\varphi$ is the random dynamical system generated by the above Ornstein-Uhlenbeck process.
\end{Remark}

\section{The main results}\label{sectionmain}
We concern the classical pitchfork bifurcation system disturbed by the rotational invariant L\'evy process $(L_{t}^{\alpha})_{t\in\R}$ defined in \eqref{levy-process},
\begin{eqnarray}\label{eq0}
\mathrm{d}X_t=(\beta X_t-X_t^3)\mathrm{d}t+\sigma\mathrm{d}L_t^{\alpha}, \quad X_0\in\mathbb{R},
\end{eqnarray}
with the drift term $b(x)=\beta x-x^3$, the bifurcation parameter $\beta\in\R$, and the noise intensity parameter $\sigma\in\R$.  It is well known that, the strong solution exists and system \eqref{eq0} generates the random dynamical system. Furthermore, there exists a unique invariant measure under random dynamical system, supported by a singleton random attractor as stated in Proposition \ref{Cllapse of random attractor}. The corresponding Lyapunov exponent depicts the stability of random dynamical system well.
\begin{Definition}{\bf(Lyapunov exponent)}\label{Ludwig Arnold}
The (forward) Lyapunov exponent of the solution $\Phi(t)x$ of a non-autonomous linear differential equation $\dot{x}_t=A(t)x_t$ starting at time $t=0$ at the state $x\in\mathbb{R}^d$ is defined to be the Lyapunov index of $\Phi(t)x$,
$$\lambda^+(x)=\lambda(x):=\limsup_{t\to+\infty}\frac{1}{t}\ln||\Phi(t)x||,$$
and for two-sided time the backward Lyapunov exponent of $\Phi(t)x$ is defined as the Lyapunov index of $\Phi(-t)x$,
$$\lambda^-(x)=\lambda(x):=\limsup_{t\to+\infty}\frac{1}{t}\ln||\Phi(-t)x||=\limsup_{t\to-\infty}\frac{1}{|t|}\ln||\Phi(t)x||.$$
\end{Definition}
If the linear cocycle $\Phi$ satisfies the integrability conditions in Multiplicative ergodic Theorem \cite{Ludwig Arnold}, the Lyapunov exponent exists as a limit and is a constant in the ergodic case.

\begin{Theorem}{\bf(Lyapunov exponent)}\label{Lya-expo}
In the case of truncated with $\alpha\in(1,2)$, the Lyapunov exponent associated with the unique invariant measure of system \eqref{eq0} is negative. 
\end{Theorem}
The singleton random attractor $\{a_{\beta}(\omega)\}_{\omega\in\Omega}$ for system \eqref{eq0} is called {\itshape locally uniformly attractive} if there exists $\delta>0$ such that 
$$\lim_{t\to0}\sup_{x\in(-\delta,\delta)}\esssup_{\omega\in\Omega}|\varphi(t,\omega,a_{\beta}(\omega)+x)-a_{\beta}(\theta_t\omega)|=0.$$
%\end{Definition}

\begin{Theorem}{\bf(Uniformly exponential attractive)}\label{uniformlyattractive}
In the cases of truncated and non-truncated with $\alpha\in(1,2)$, the unique attracting random equilibrium $\{a_{\beta}(\omega)\}_{\omega\in\Omega}$ for system \eqref{eq0} satisfies,\\
(i) For $\beta<0$, the random attractor $\{a_{\beta}(\omega)\}_{\omega\in\Omega}$ is globally uniformly exponential attractive, $i.e.$
\begin{eqnarray}\label{th1.1}
|\varphi(t,\omega,x)-\varphi(t,\omega,a_{\beta}(\omega))|<e^{\beta t}|x-a_{\beta}(\omega)|, \quad\mbox{for all }x\in\R.
\end{eqnarray}
(ii) For $\beta>0$, the random attractor $\{a_{\beta}(\omega)\}_{\omega\in\Omega}$ is not locally uniformly attractive.
\end{Theorem}

For both truncated and non-truncated cases, define the finite-time Lyapunov exponent associate with the invariant measure $\delta_{a_{\beta}(\omega)}$ for the system $\eqref{eq0}$ on the compact time interval $[0,T]$ as follows \cite{Mark Callaway},
$$\lambda_\beta^{T,\omega}:=\frac{1}{T}\ln\bigg|\frac{\partial\varphi_\beta}{\partial x}(T,\omega,a_\beta(\omega))\bigg|.$$
\begin{Remark}\label{Lyapunov}
In the case of truncated with $\alpha\in(1,2)$, the linear cocycle $\Phi$ satisfies the integrability conditions in Multiplicative ergodic Theorem. Hence the Lyapunov exponent exists as a limit,
$$\lambda_\beta^{\infty}=\lim_{T\to\infty}\lambda_\beta^{T,\omega}.$$
In fact, Let $\Phi_\beta(t,\omega):=\frac{\partial\varphi_\beta}{\partial x}(t,\omega,a_\beta(\omega))$ be the linearized random dynamical system along the random equilibrium $a_\beta(\omega)$. The linearized equation along the random equilibrium $a_\beta(\omega)$ is given by
\begin{eqnarray}\label{linearized}
\dot{\xi_t}=(\beta-3a_\beta(\theta_t\omega)^2)\xi_t, \quad \xi_0\in\mathbb{R}.
\end{eqnarray}
Thus 
\begin{eqnarray}\label{linearizedflow}
\Phi_\beta(t,\omega)
=\exp\bigg(\int_0^t(\beta-3a_\beta(\theta_s\omega)^2))\mathrm{d}s\bigg).
\end{eqnarray}
Further, 
\begin{eqnarray*}
\mathbb{E}\bigg(\sup_{0\leq t\leq1}\log^+||\Phi_\beta(t,\omega)^{\pm1}||\bigg)
&=&\mathbb{E}\bigg(\sup_{0\leq t\leq1}\int_0^t(\beta-3a_\beta(\theta_s\omega)^2)^\pm\mathrm{d}s\bigg)\\
&\leq&\mathbb{E}\big(\sup_{0\leq t\leq1}\big|\beta t-3\int_0^t a_\beta(\theta_s\omega)^2\mathrm{d}s\big|\big)\\
&\leq&|\beta|+3\mathbb{E}\big(\int_0^1a_\beta(\theta_s\omega)^2\mathrm{d}s\big)\\
&=&|\beta|+3\int_0^1\mathbb{E}a_\beta(\theta_s\omega)^2\mathrm{d}s\\
&<&+\infty,
\end{eqnarray*}
where we have used Lemma \ref{moment} and the Euclidean norm $||\cdot||$ in $\mathbb{R}$. Hence the linear cocycle $\Phi$ satisfies the integrability conditions in the truncated case.

While In the case of non-truncated with $\alpha\in(1,2)$, we can obtain that the linear cocycle $\Phi$ doesn't satisfy the integrability conditions in Multiplicative ergodic Theorem through similar computation. Whether the Lyapunov exponent exists as a limit can not assured. While we can still consider the finite-time Lyapunov exponent. Thus the finite-time Lyapunov exponent points out the limitations and sheds new light on the classical Lyapunov exponent.
%\begin{eqnarray*}
%\lim_{t\to\pm\infty}\frac{1}{t}\log||\Phi_\beta(t,\omega)x||
%&=&\lim_{t\to\pm\infty}\frac{|x|}{t}\int_0^t(\beta-3a_\beta(\theta_s\omega)^2))\mathrm{d}s\\
%&=&
%\end{eqnarray*}
\end{Remark}
\begin{Theorem}{\bf(Finite-time Lyapunov exponent)}\label{Fini-tim Lya exponent}
In the cases of truncated and non-truncated with $\alpha\in(1,2)$, the finite-time Lyapunov exponent $\lambda_\beta^{T,\omega}$ associated with $\{a_\beta(\omega)\}_{\omega\in\Omega}$ for the system $\eqref{eq0}$ satisfies the following results,\\
(i)For $\beta<0$, the random attractor $\{a_\beta(\omega)\}_{\omega\in\Omega}$ is finite-time attractive for any $T>0$, $i.e.$
$$\lambda_\beta^{T,\omega}\leq\beta<0 \quad\mbox{ for all }\omega\in\Omega.$$
(ii)For $\beta>0$, the random attractor $\{a_\beta(\omega)\}_{\omega\in\Omega}$ is not finite-time attractive for any $T>0$, $i.e.$
$$\mathbb{P}\{\omega\in\Omega:\lambda_\beta^{T,\omega}>0\}>0.$$
\end{Theorem}
Given a linear random dynamical system $(\theta,\varphi)$, there exists a corresponding matrix-valued function $\Phi:\mathbb{T}\times\Omega\to\mathbb{R}^{d\times d}$ with $\Phi(t,\omega)x=\varphi(t,\omega)x$ for all $t\in\mathbb{T}$, $\omega\in\Omega$ and $x\in\mathbb{R}^{d}$.
\begin{Definition}{\bf(Invariant projector)}\cite{Mark Callaway}
An \textit{invariant projector} of $(\theta,\Phi)$ is a measurable function $P: \Omega\to\mathbb{R}^{d\times d}$ with 
$$P(\omega)=P(\omega)^2\quad\mbox{and }P(\theta_t\omega)\Phi(t,\omega)=\Phi(t,\omega)P(\omega) \quad\mbox{for all }t\in\mathbb{T}\quad\mbox{and }\omega\in\Omega.$$
\end{Definition}

\begin{Definition}{\bf(Exponential dichotomy)}\cite{Mark Callaway}
Let $(\theta,\Phi)$ be a linear random dynamical system, and let $\gamma\in\mathbb{R}$, and $P_{\gamma}:\Omega\to\mathbb{R}^{d\times d}$ be an invariant projector of $(\theta,\Phi)$. Then $(\theta,\Phi)$ is said to admit an exponential dichotomy with growth rate $\gamma\in\mathbb{R}$, constants $\alpha>0$, $K\geq1$ and projector $P_{\gamma}$ if for almost all $\omega\in\Omega$, one has 
$$||\Phi(t,\omega)P_{\gamma}(\omega)||\leq Ke^{(\gamma-\alpha)t}\qquad\mbox{for all }t\geq0,$$
$$||\Phi(t,\omega)(1-P_{\gamma}(\omega))||\leq Ke^{(\gamma+\alpha)t}\qquad\mbox{for all }t\leq0.$$
\end{Definition}

\begin{Definition}{\bf(Dichotomy spectrum)}\cite{Mark Callaway}
Consider the linear random dynamical system $(\theta,\Phi)$. Then the dichotomy spectrum of $(\theta,\Phi)$ is defined by 
$$\Sigma:=\{\gamma\in\mathbb{R}: (\theta,\Phi) \mbox{ does not admit an exponential dichotomy with growth rate } \gamma\}.$$
\end{Definition}

\begin{Theorem}{\bf(Dichotomy spectrum)}\label{Dichotomy spectrum}
Let $\Phi_\beta(t,\omega):=\frac{\partial\varphi_\beta}{\partial t}(t,\omega,a_\beta(\omega))$ be the linearized random dynamical system about system \eqref{eq0} along the random equilibrium $a_\beta(\omega)$. In the cases of truncated and non-truncated with $\alpha\in(1,2)$, the dichotomy spectrum $\Sigma_\beta$ of $\Phi_\beta$ is given by 
$$\Sigma_\beta=[-\infty,\beta] \quad \mbox{for all }\beta\in\mathbb{R}.$$
\end{Theorem}
We prove the main results in Section \ref{Proof of the main results}.

\section{Basic analysis on the pitchfork bifurcation model}\label{sectionbasic}
There exists an invariant measure under the Markov semigroup $P_{t}$ associated with system \eqref{d-dim-sys} as stated in Proposition \ref{Exponentially ergodic}. If the initial value $X_0$ is distributed as $\mu$. Then for any $t > 0$, the distribution of $X_t$ is $\mu P_t$. We denote the space of probability measures on $\mathbb{R}$ as $\mathcal{P}$.
Given two probability measures $\mu$ and $\nu$ on $\mathbb{R}$, the standard $L^p$ -Wasserstein distance $W_p$ for all $p\in[1,+\infty)$ (with respect to the Euclidean norm $|\cdot|$) is given by
$$W_p(\mu,\nu)=\inf_{\mathcal{C}(\mu, \nu)}\bigg(\int_{\mathbb{R}\times\mathbb{R}}|x-y|^p\mathrm{d}\Pi(x,y)\bigg)^{\frac{1}{p}},$$
where $\mathcal{C}(\mu, \nu)$ is the collection of measures on $\mathbb{R} \times \mathbb{R}$ having $\mu$ and $\nu$ as marginals. 
\begin{Proposition}{\bf(Exponentially ergodic)}\label{Exponentially ergodic}
In the case of non-truncated with $\alpha\in(1,2)$, the Markov semigroup for the system \eqref{eq0} admits a unique invariant measure $\rho$. Furthermore, for initial distribution $\tilde{\rho}\in\mathcal{P}$, $t\geq0$ and $p\geq1$, there exist positive constants $K$ and $c$, such that,
\begin{eqnarray}\label{ergodic}
W_p(\tilde{\rho} P_t,\rho) \leq Ke^{-ct}W_p(\tilde{\rho},\rho).
\end{eqnarray}
\end{Proposition}
\begin{proof}
Note that $x^2+xy+y^2\geq\frac{1}{4}(x-y)^2$. There exist positive constants $L_0$, $K_1$, $K_2$ and $\theta\geq2$, such that the drift term satisfies 
\begin{eqnarray*}
(b(x)-b(y))(x-y)\leq\left\{\begin{array}{ll}
K_1|x-y|^2, \qquad  |x-y|\leq L_0;\\
-K_2|x-y|^\theta,  \quad\, |x-y|>L_0.
\end{array}\right.
\end{eqnarray*}
According to Theorem 1.2 \cite{Jian Wang}, for any $x,y\in\mathbb{R}$, 
$$W_p(\delta_xP_t,\delta_yP_t)\leq Ke^{-ct}W_p(\delta_x,\delta_y).$$ 
Furthermore, refer to the proof of Theorem 3.2 in \cite{Martin Friesen Peng Jin Jonas Kremer and Barbara Rudiger}, for any $\tilde{\rho},\tilde{\tilde{\rho}}\in\mathcal{P}$, 
$$W_p(\tilde{\rho}P_t,\tilde{\tilde{\rho}}P_t)\leq Ke^{-ct}W_p(\tilde{\rho},\tilde{\tilde{\rho}}).$$ 
Refer to Theorem 1.2 and Theorem 3.2 in \cite{Martin Hairer and Jonathan C. Mattingly}, the invariant measure $\rho$ exists uniquely and satisfies \eqref{ergodic}.
\end{proof}
By revising the proof of Theorem 1.2 \cite{Jian Wang}, we can obtain the following result similarly.
\begin{Corollary}
In the case of truncated with $\alpha\in(1,2)$, the Markov semigroup for the system \eqref{eq0} admits a unique invariant measure $\rho$. Furthermore, for initial distribution $\tilde{\rho}\in\mathcal{P}$, $t\geq0$ and $p\geq1$, there exist positive constants $K$ and $c$, such that,
\begin{eqnarray}
W_p(\tilde{\rho} P_t,\rho) \leq Ke^{-ct}W_p(\tilde{\rho},\rho).
\end{eqnarray}
\end{Corollary}

There exists a unique invariant measure under the random dynamical system generated by system \eqref{d-dim-sys}, which is supported by a random attractor consisting of a singleton, as stated in Proposition \ref{Cllapse of random attractor}.

\begin{Definition}{\bf(Pullback random attractor)}\cite{Hans Crauel and Franco Flandoli, Michael Rockner}
A random attractor for an RDS $\varphi$ is a compact random set $A$ satisfying $\mathbb{P}$-a.s.:\\
(i)$A$ is strictly invariant, i.e. $\varphi(t,\omega)A(\omega)=A(\theta_t\omega)$ for all $t>0$.\\
(ii)$A$ attracts all deterministic bounded sets $B\subset X$, i.e.
$$\lim_{t\to\infty}d(\varphi(t,\theta_{-t}\omega)B,A(\omega))=0.$$ 
\end{Definition}

\begin{Proposition}{\bf(Cllapse of random attractor)}\label{Cllapse of random attractor}
In the cases of truncated and non-truncated with $\alpha\in(1,2)$, system \eqref{eq0} exists a compact random attractor $\mathcal{A}(\omega)$ consisting of a singleton, that is 
$$\mathcal{A}(\omega)=\{a(\omega)\}.$$ Furthermore, the invariant Markov measure $\mu$ with $\mu_{\omega}=\delta_{a(\omega)}$ is the unique invariant measure under the random dynamical system $\varphi$ generated by system \eqref{eq0}.
\end{Proposition}
\begin{proof}
The proof of existence for the random attractor is standard, one can refer to \cite{Michael Rockner}. If the random attractor is singleton, the proof for uniqueness on invariant measure $\mu$ refer to Corollary 3.5 \cite{Hans Crauel and Franco Flandoli}. 

The proof of singleton is similar to Theorem 3.1 \cite{Hans Crauel and Franco Flandoli}, based on the uniqueness of the invariant measure to the Markov semigroup for system \eqref{eq0} and the order preserving of the random dynamical system. The former condition is confirmed by Proposition \ref{Exponentially ergodic}. Now we prove the latter requirement. For both truncated and non-truncated cases, consider the system 
$$\mathrm{d}Y_t=\beta Y_t\mathrm{d}t+\sigma\mathrm{d}L_t^{\alpha}, \quad \mbox{in } \mathbb{R}.$$
Let $u_t=X_t-Y_t$. Then $u_t$ is continuous and 
\begin{eqnarray}\label{Lm1.1}
\mathrm{d}u_t=\beta u_t\mathrm{d}t-(u_t+Y_t)^3.
\end{eqnarray}
Then the system \eqref{Lm1.1} generates a continuous perfect cocycle. According to the Remark $2.3$ \cite{Hans Crauel and Franco Flandoli} or Theorem 1.8.4 \cite{Ludwig Arnold}, system \eqref{Lm1.1} is monotonicity.

For any $x_1,x_2\in\mathbb{R}$ satisfing $x_1>x_2$, take $u_1=x_1-y_0$ and $u_2=x_2-y_0$. Then $u_1>u_2$ and $u(t,\omega,u_1)>u(t,\omega,u_2)$ almost surely, which implys $u(t,\omega,u_1)+Y_t>u(t,\omega,u_2)+Y_t$. Thus we obtain that 
$x_1>x_2$ implies $\varphi(t,\omega,x_1)>\varphi(t,\omega,x_2)$, for all $x_1,x_2\in\mathbb{R}$ almost surely.
That is for both truncated and non-truncated cases, solutions of system \eqref{eq0} are monotonicity.
\end{proof}

%Recall that the density function $p(t,x)$ of diffusion process $X_t$ is the fundamental solution of Fokker-Planck equation \eqref{fokker-planck} for system \eqref{eq0}, i.e.
%$$p_t=-(bp)^{'}+|\sigma|^\alpha\int_{|z|<1}[p(x+z)-p(x)]\nu(\mathrm{d}z).$$
%The associated Markov semigroup of system \eqref{eq0} exists an unique invariant measure $\rho$ precisely given by the  invariant Markov measure $\delta_{a(\omega)}$. Specifically, $\rho=\mathbb{E}\delta_{a(\omega)}$, i.e. 
%$$\rho(B)=\int_{\Omega}\delta_{a(\omega)}(B)\mathrm{d}P(\omega).$$
%The density $p$ of the invariant measure $\rho$ satisfies the elliptic equation $\mathcal{L}u=0$. That is 
%\begin{eqnarray}\label{pfeq1}
%-(bp)^{'}+|\sigma|^\alpha\int_{|z|<1}[p(x+z)-p(x)]\nu(\mathrm{d}z)=0.
%\end{eqnarray}
%It does exists a unique strong solution.
Note that, the relation between the invariant measure $\rho$ under the Markov semigroup and the invariant measure $\mu$ under the random dynamical system $\varphi$ refer to Remark \ref{relation invariant measure}. In order to analyze the Lyapunov exponent, it is necessary to estimate the moment of the invariant measure $\rho$ under the Markov semigroup.

\begin{Lemma}\label{moment}{\bf (Finite moment)}
In the case of truncated with $\alpha\in(1,2)$, the invariant measure $\rho$ of the Markov semigroup for system \eqref{eq0} admits finite any order moment estimate. That is 
\begin{eqnarray*}
\int_{\mathbb{R}}|x|^n\rho(\mathrm{d}x)\leq c_n, \mbox{ for any } n\in\mathbb{Z^+},
\end{eqnarray*}
with positive constants $c_n$.
%with $c=\frac{1}{\lambda}\bigg[\frac{1}{8}(\lambda+2\beta)^2+\frac{2c_\alpha \sigma^2}{2-\alpha}\bigg]$.
\end{Lemma}
\begin{proof}
For any $\lambda>0$, $t>0$, It$\hat{o}$'s formular implys that, 
\begin{eqnarray*}
\mathrm{d}[X^2(t)]&=&2X(t)\big[\beta X(t)-X^3(t)\big]\mathrm{d}t
+\int_{|z|<1}\bigg\{\big[X(t-)+\sigma z\big]^2-X^2(t-)\bigg\}\widetilde{N}(\mathrm{d}t,\mathrm{d}z)\\
& &+\int_{|z|<1}\big[(X(t-)+\sigma z)^2-X^2(t-)-2\sigma zX(t-)\big]\nu(\mathrm{d}z)\mathrm{d}t\\
&=& -\lambda X^2(t)\mathrm{d}t+\bigg[(\lambda+2\beta) X^2(t)-2X^4(t)+\int_{|z|<1}\sigma^2z^2\nu(\mathrm{d}z)\bigg]\mathrm{d}t\\
&&+\int_{|z|<1}\big[2\sigma zX(t-)+\sigma^2z^2\big]\widetilde{N}(\mathrm{d}t,\mathrm{d}z).
\end{eqnarray*}
The solution to the above equation is given by
\begin{eqnarray*}
X^2(t)&=&e^{-\lambda t}x^2+\int_{0}^{t}e^{-\lambda (t-s)}\bigg[(\lambda+2\beta) X^2(t)-2X^4(t)+\int_{|z|<1}\sigma^2z^2\nu(\mathrm{d}z)\bigg]\mathrm{d}s\\
&&+\int_{0}^{t}\int_{|z|<1}e^{-\lambda (t-s)}\big[2\sigma zX(s-)+\sigma^2z^2\big]\widetilde{N}(\mathrm{d}s,\mathrm{d}z)
\end{eqnarray*}
Taking expectation on both sides derives that,
\begin{eqnarray*}
\mathbb{E}^x\big[X^2(t)\big]
&\leq&
e^{-\lambda t}x^2+\frac{1}{\lambda}\bigg[\frac{1}{8}(\lambda+2\beta)^2+\frac{2c_\alpha \sigma^2}{2-\alpha}\bigg].
\end{eqnarray*}
Denote $c=\frac{1}{\lambda}\bigg[\frac{1}{8}(\lambda+2\beta)^2+\frac{2c_\alpha \sigma^2}{2-\alpha}\bigg]$. Based on the invariance of $\rho$, and following the thought in \cite{Sandra Cerrai}, we obtain 
\begin{eqnarray*}
\int_{\mathbb{R}}x^2\rho(\mathrm{d}x)
=\int_{\mathbb{R}}P_{t}x^2\rho(\mathrm{d}x)
=\int_{\mathbb{R}}\mathbb{E}^x X^2(t)\rho(\mathrm{d}x)
\leq e^{-\lambda t}\int_{\mathbb{R}}x^2\rho(\mathrm{d}x)+c.
\end{eqnarray*}
Hence,
\begin{eqnarray*}
\int_{\mathbb{R}}x^2\rho(\mathrm{d}x)\leq c.
\end{eqnarray*}

It can be proved similarly that any even order moment are finite with respect to the invariant measure $\rho$. 
The result follows.
\end{proof}
\begin{Remark}
We have to consider the truncated case. In fact, Lemma \ref{moment} is the basis for computing Lyapunov exponential, referring the proof of Theorem \ref{Lya-expo} for details. While for the non-truncated case, the above property doesn't hold any more. In this case,
$$\int_{\mathbb{R}}|x|^p\rho(\mathrm{d}x)\quad\left\{\begin{array}{ll}
        < \infty, \mbox{ for } p\in(0,\alpha),\\
        \\
        =\infty, \mbox{ for } p\in[\alpha,2),\\
\end{array}
\right.$$
due to the property of $\alpha$-stable L\'evy process. 
\end{Remark}

\section{The key Lemmas}\label{sectionkey}

Consider the operator $\mathscr{L}$:
$$\mathscr{L}f(x)=\mathcal{L}f(x)-b(x)\cdot\triangledown f(x)+c(x)f(x),$$
for any $f\in \mathcal{B}_b(\mathbb{R}^d)$, where $\mathcal{L}$ is defined in \eqref{generator}, $b(x)$ and $c(x)$ are continuous.
\begin{Lemma}\label{maximum principle}{\bf(Strong maximum principle)}  In the cases of truncated and non-truncated, suppose there exists some $r\in(0,+\infty)$ such that $\mathscr{L}f\leq0$ in $B_{r}$, with $c(x)\leq0$ in $B_{r}$. Further, $f\geq0$ in $B_{r}^{c}$, then $f>0$ in $B_{r}$, unless $f$ vanishes identically in $\mathbb{R}^d$.
%For truncated case, suppose there exists some $r\in(0,+\infty)$ such that $\mathscr{L}f\leq0$ in $B_{r}$, with $c(x)\leq0$ in $B_{r}$. Further, if $f\geq0$ in $B_{r}^{c}$, then $f>0$ in $B_{r}$, unless $f$ vanishes identically in $B_{r+1}$.
\end{Lemma}
\begin{proof}
{\bf The first step.} We prove a weak version for both non-truncated and truncated cases. That is the following result holds.

Suppose there exists some $r\in(0,+\infty)$ such that $\mathscr{L}f\leq0$ in $B_{r}$, with $c(x)\leq0$ in $B_{r}$. If $f\geq0$ in $B_{r}^{c}$, then $f\geq0$ in $\mathbb{R}^d$. 

Suppose by contradiction that the minimal point $x_{0}\in B_{r}$ satisfies $f(x_{0})<0$ and $\triangledown f(x_{0})=0$. It is the global minimum in $\mathbb{R}^d$, thus $f(x_{0}+y)-f(x_{0})\geq0$ for any $y\in \mathbb{R}^d$. 

For the non-truncated case, we have that $$\mathscr{L}f(x_{0})>0,$$
which leads to a contradiction. In fact,
\begin{eqnarray*}
\mathscr{L}f(x_{0})&=&\mathcal{L}f(x_{0})-b(x_{0})\cdot\triangledown f(x_{0})+c(x_{0})f(x_{0})\\
&=&\int_{\mathbb{R}^{d}}[f(x_{0}+z)-f(x_{0})]\nu(\mathrm{d}z)+c(x_{0})f(x_{0})\\
&\geq&\int_{B_{2r}^{c}}[f(x_{0}+z)-f(x_{0})]\nu(\mathrm{d}z)+c(x_{0})f(x_{0})\\
&\geq&\int_{B_{2r}^{c}}[f(x_{0}+z)-f(x_{0})]\nu(\mathrm{d}z)\\
&\geq&-\int_{B_{2r}^{c}}f(x_{0})\nu(\mathrm{d}z)\\
&>& 0,
\end{eqnarray*}
where we have used the fact that $f(x_{0})<0$, $c(x)\leq0$ in $B_{r}$, and $f(x_{0}\pm z)\geq0$ for $z\in B_{2r}^{c}$. The last relation holds due to $f(x)\geq0$ in $B_{r}^{c}$, and $x_{0}\pm z\in B_{r}^{c}$ for $z\in B_{2r}^{c}$.

For the truncated case, we can also conclude a contradiction. In fact,
\begin{eqnarray*}
\mathscr{L}f(x_{0})&=&\mathcal{L}f(x_{0})-b(x_{0})\cdot\triangledown f(x_{0})+c(x_{0})f(x_{0})\\
&=&\int_{B_{1}}[f(x_{0}+z)-f(x_{0})]\nu(\mathrm{d}z)+c(x_{0})f(x_{0})\\
&\geq&\int_{B_{1}}[f(x_{0}+z)-f(x_{0})]\nu(\mathrm{d}z)\\
&\geq&0,
\end{eqnarray*}
with the fact that $x_0$ is the global minimum, $f(x_0)<0$ and $c(x_0)\leq0$. We get that 
$$\int_{B_{1}}[f(x_{0}+z)-f(x_{0})]\nu(\mathrm{d}z)=0.$$
This assures 
$$f(x)=f(x_0)<0, \quad\mbox{ for any } \quad x\in B_1(x_0).$$
%Once $B_r^c\cap B_1(x_0)\neq\emptyset$, we arrive at a contradiction.
For any $x\in B_1(x_0)$, repeat the above procedure, we get 
$$f(x)=f(x_0)<0, \quad\mbox{ for any } \quad x\in B_2(x_0).$$
Continuing the above procedure, we will arrive at a contradiction eventually with $f(x)=0$ on $\mathbb{R}^d$.

{\bf The second step.} Having the fact that $f\geq0$ in the whole of $\mathbb{R}^{d}$ for both truncated and non-truncated case. We prove the strictly positivity in this step. 

If $f$ is not strictly positive in $B_{r}$, the minimal point $x_{0}\in B_{r}$ satisfies $f(x_{0})=0$ and $\triangledown f(x_{0})=0$. 

For the non-truncated case, 
$$\int_{\mathbb{R}^{d}}f(x_{0}+z)\nu(\mathrm{d}z)=0.$$ 
In fact, 
\begin{eqnarray*}
\mathscr{L}f(x_{0})&=&\mathcal{L}f(x_{0})-b(x_{0})\cdot\triangledown f(x_{0})+c(x_{0})f(x_{0})\\
&=&\int_{\mathbb{R}^{d}}f(x_{0}+z)\nu(\mathrm{d}z)\\
&\geq&0,
\end{eqnarray*}
by the first step, $f\geq0$ in $\mathbb{R}^{d}$. While $\mathscr{L}f(x)\leq0$ in $B_{r}$. The integral above must be vanished identically. As a result,  the integrated function $f$  vanished almost everywhere in the whole space $\mathbb{R}^{d}$ with respective to Lebesgue measure. Thus $f$ vanished identically in $\mathbb{R}^{d}$ due to its continuity.

For the truncated case,
$$\int_{B_1}f(x_{0}+z)\nu(\mathrm{d}z)=0.$$ 
In fact, 
\begin{eqnarray*}
\mathscr{L}f(x_{0})&=&\mathcal{L}f(x_{0})-b(x_{0})\cdot\triangledown f(x_{0})+c(x_{0})f(x_{0})\\
&=&\int_{B_1}f(x_{0}+z)\nu(\mathrm{d}z)\\
&\geq&0,
\end{eqnarray*}
while $\mathscr{L}f(x)\leq0$ in $B_{r}$. Hence, the integral above must be vanished identically. By the first step, $f\geq0$ in $\mathbb{R}^{d}$, as a result 
$$f(x)=f(x_0)=0, \quad\mbox{ for any } \quad x\in B_1(x_0).$$
%Once $B_r=B_1(x_0)$, we have $f(x)=0$ in $B_r$. Repeat the above procedure, we obtain that $f(x)=0$ in $B_{r+1}$.
%Or $B_r\neq B_1(x_0)$, by repeating the above procedure for $x\in B_{r}\cap B_1(x_0)$, we obtain that $f(x)=0$ for $x\in B_{r+1}\cap B_2(x_0)$. Continue the procedure for $x\in B_{r}\cap B_2(x_0)$ and so on. After several times we will arrive at the result $f(x)=0$ in $B_{r+1}$ eventually. 

Repeat the above procedure for $x\in B_1(x_0)$, we can obtain that $f(x)=0$ in $B_2(x_0)$. In fact, we arrive at the result $f(x)=0$ on $\mathbb{R}^d$ eventually.
\end{proof}

Let  $\rho(\mathrm{d}x)=p(x)\mathrm{d}x$  be the invariant measure of system \eqref{d-dim-sys}, i.e.
$$\mathcal{L}_{b}^{*}p=0,\,\,p\geq0,\,\,\int_{\mathbb{R}^{d}}p(x)\mathrm{d}x=1.$$
\begin{eqnarray}\label{pfeq1}
\mathcal{L}_{b}^{*}p=\left\{\begin{array}{ll}
-\triangledown\cdot(bp)+|\sigma|^\alpha\int_{\mathbb{R}}[p(x+z)-p(x)]\nu(\mathrm{d}z), & \mbox{ if} \quad\nu(B_1^c)\neq0,\\
\\
-\triangledown\cdot(bp)+|\sigma|^\alpha\int_{|z|<1}[p(x+z)-p(x)]\nu(\mathrm{d}z), & \mbox{ if} \quad\nu(B_1^c)=0.
                                      \end{array}
                                      \right.
\end{eqnarray}

\begin{Lemma}\label{levyproperty1}{\bf(Support of stationary density)}
In the cases of truncated and non-truncated, the stationary density $p$ corresponding to the invariant measure $\rho$ is strictly positive in the whole space $\mathbb{R}$ for system \eqref{eq0}.  
\end{Lemma}
\begin{proof}
Note that the stationary density $p$ satisfies $\mathscr{L}p=0$, $p\geq0$, $\int_{\mathbb{R}^{d}}p(x)\mathrm{d}x=1$, with $c(x)=-\triangledown\cdot b(x)$ and $b(x)=\beta x-x^{3}$. Due to $c(x)\leq0$ is not always satisfied, we write $c(x)=c(x)^{+}-c(x)^{-}$ with the positive and negative part  $c(x)^{+}$ and $c(x)^{-}$. Define 
$$\tilde{\mathscr{L}}p(x):=\mathscr{L}p(x)-c^{+}(x)p(x).$$
Then $p$ satisfies
\begin{eqnarray*}
\tilde{\mathscr{L}}p(x)&=&\mathcal{L}p(x)-b(x)\cdot\triangledown p(x)+c(x)p(x)-c^{+}(x)p(x)\\
&=&-c^{+}(x)p(x)\leq0.
\end{eqnarray*}
For both truncated and non-truncated cases, it satisfies the condition of Lemma \ref{maximum principle} for any $r\in(0,+\infty)$. Suppose $p$ is not strictly positive in $\mathbb{R}$, then there exists $x_{0}$ such that $p(x_{0})=0$. Using the lemma \ref{maximum principle} for $\tilde{\mathscr{L}}$ with any $r\in(|x_0|,+\infty)$, we have $p\equiv0$ in $\mathbb{R}$. This leads to a contradiction. 
%In fact, for the non-truncated case, we get $p\equiv0$ in $\mathbb{R}$ directly. For the truncated case, we can get $p\equiv0$ in $B_{r+1}$ for any $r\in(|x_0|,+\infty)$ firstly. And then $p\equiv0$ in $\mathbb{R}$ also holds, due to the arbitrariness of $r$.
\end{proof}

\begin{Remark}
About the strong maximum principle, one can refer to \cite{Claudia Bucur, Qing Han}. 
\end{Remark}

Define the stochastic process
\begin{eqnarray}\label{levy-drift}
\widetilde{L}_t=L_t+g(t), 
\end{eqnarray}
with the process $L_t$ in \eqref{levy-process}, the drift, $g_t\in C(\mathbb{R},\mathbb{R}^d)$, and $g_0=0$. Then the density is given by 
\begin{eqnarray}\label{density-levy-drift}
p_{\widetilde L}(t,x)=p_{L}(t,x)*\delta_{g}(t,x)
%&=&\int_{\mathbb{R}}p_{L}(t,y)\delta_{g}(t,x-y)\mathrm{d}y\\
%&=&\int_{\mathbb{R}}p_{L}(t,y)\delta(x-g_t-y)\mathrm{d}y\\
%&=&\int_{\mathbb{R}}p_{L}(t,y)\delta(y-(x-g_t))\mathrm{d}y\\
=p_{L}(t,x-g_t).
\end{eqnarray}
Consider the position of the Bilateral supreme process $\underset{0\leq t\leq T}{\sup}|\widetilde{L}_t|$. Based on the thought of Taylor \cite{S.J.Taylor}, Doob inequality, stochastic continuity and the positivity of transition probability, we have the following result, which is similar to but strong than the irreducibility.

\begin{Lemma}\label{levyproperty2}{\bf(Bilateral suprema)} 
In the cases of truncated and non-truncated, the process $\widetilde{L}_t$ defined in \eqref{levy-drift} satisfies $\mathbb{P}(\underset{0\leq t\leq T}{\sup}|\widetilde{L}_t|<\varepsilon)>0$,  for any $\varepsilon>0$, $T>0$.
\end{Lemma}
\begin{proof}
The proof idea refer to Taylor \cite{S.J.Taylor}. The process $\widetilde{L}_t$ is stochastically continuous, which implys $\widetilde{L}_t$ convergence to $\widetilde{L}_0$ in distribution as $t$ goes to $0$ on one hand. For any $\delta_1>0$, $\delta_0>0$, there exists $T^{'}>0$, s.t. for all $t<T^{'}(>0)$,
$$\mathbb{P}^{x}(|\widetilde{L}_t|<\delta_0)=\mu_{\widetilde{L}_t}(B_{\delta_0})\geq\mu_{\widetilde{L}_0}(B_{\delta_0})-\delta_1,$$
where $B_{\delta_0}$ is the ball with radius $\delta_0$ and centered at $0$. 

On the other hand, there exists a subsequence $\{t_n\}$ $(n=1,2,\cdots)$ tends to $0$ as $n$ goes to $\infty$, such that $\widetilde{L}_{t_n}$ convergence to $\widetilde{L}_0$ almost surely as $n$ goes to $\infty.$
This derives that, for any $\delta_2>0$, there exists $T^{''}>0$, such that  $t_n<T^{''}(>0)$ for $n$ large enough, and $$\mathbb{E}|\widetilde{L}_{t_n}|\leq\mathbb{E}|\widetilde{L}_{0}|+\delta_2.$$
 
Consequently, under the condition $x\in B_{\delta_0}$, for any $\varepsilon>0$ and the above $\delta_0$, $\delta_1$, $\delta_2$, there exists $T_1>0(\ll1)$ such that,
$$\mathbb{P}^{x}(|\widetilde{L}_{T_1}|<\delta_0)\geq 1-\delta_1,$$
and 
$$\mathbb{E}|\widetilde{L}_{T_1}|\leq\mathbb{E}|\widetilde{L}_{0}|+\delta_2\leq\delta_0+\delta_2.$$
Futhermore, for the above parameters, Doob inequality assures that,
$$\mathbb{P}^{x}(\sup_{0\leq t\leq T_1}|\widetilde{L}_t|\geq\varepsilon)\leq\frac{1}{\varepsilon}\mathbb{E}|\widetilde{L}_{T_1}|\leq\frac{1}{\varepsilon}(\delta_0+\delta_2).$$
Take $\delta_0=\frac{\varepsilon}{4}$, $\delta_1=\frac{1}{4}$, $\delta_2=\frac{\varepsilon}{4}$, we obtain
$$\mathbb{P}^x(\sup_{0\leq t\leq T_1}|\widetilde{L}_t|<\varepsilon, |\widetilde{L}_{T_1}|<\delta_0)\geq\frac{1}{4}.$$
Because of the smoothness of the density by Theorem 28.4 Sato \cite{Ken-Iti Sato} and positivity of continuous density in Sharpe \cite{Michael Sharpe}, for any $T>0$, take $k\in \mathbb{N}_+$ such that $T_{k-1}<T\leq T_{k}$ with $T_{i}=iT_1$ $(i=0,1,2,\cdots)$, the finite dimensional distribution satisfies that, 
$$\mathbb{P}^x(\widetilde{L}_{T_1}\in B_{\delta_0},\widetilde{L}_{T_2}\in B_{\delta_0},\cdots,\widetilde{L}_{T_{k-1}}\in B_{\delta_0})>0.$$
%The strong Markov property \cite{Ioannis Karatzas and Steven E. Shreve} derives that, for some $y\in B^{\delta_0}$ the following holds,
%\begin{eqnarray*}
%&&\mathbb{P}^{x}(\sup_{0\leq t\leq T_2}|\widetilde{L}_t|<\varepsilon)\\
%&\geq&\mathbb{P}^{x}\bigg(\sup_{0\leq t\leq T_1}|\widetilde{L}_t|<\varepsilon,\widetilde{L}_{T_1}\in B^{\delta_0},\sup_{T_1\leq t\leq T_2}|\widetilde{L}_t|<\varepsilon,\widetilde{L}_{T_2}\in B^{\delta_0}\bigg)\\
%&=&\mathbb{E}^x\bigg[\mathbb{P}^{x}\bigg(\sup_{0\leq t\leq T_1}|\widetilde{L}_t|<\varepsilon,\widetilde{L}_{T_1}\in B^{\delta_0},\sup_{T_1\leq t\leq T_2}|\widetilde{L}_t|<\varepsilon,\widetilde{L}_{T_2}\in B^{\delta_0}\bigg|\widetilde{L}_{T_1}=y\bigg)\bigg]\\
%&=&\mathbb{P}^{x}\bigg(\sup_{0\leq t\leq T_1}|\widetilde{L}_t|<\varepsilon,\widetilde{L}_{T_1}\in B^{\delta_0}\bigg)\mathbb{P}^{y}\bigg(\sup_{0\leq t\leq T_1}|\widetilde{L}_t|<\varepsilon,\widetilde{L}_{T_1}\in B^{\delta_0}\bigg)\\
%&\geq&(\frac{1}{4})^2,
%\end{eqnarray*}
The strong Markov property \cite{Ioannis Karatzas and Steven E. Shreve} derives that, for some $(y_0,y_1,\cdots,y_{k-1})\in\{x\}\times \underbrace{B_{\delta_0}\times \cdots\times B_{\delta_0}}_{k-1}$, 
\begin{eqnarray*}
&&\mathbb{P}^{x}(\sup_{0\leq t\leq T}|\widetilde{L}_t|<\varepsilon)\\
&\geq&\mathbb{P}^{x}(\sup_{0\leq t\leq T_1}|\widetilde{L}_t|<\varepsilon,\widetilde{L}_{T_1}\in B_{\delta_0},\cdots,\sup_{T_{k-1}\leq t\leq T_k}|\widetilde{L}_t|<\varepsilon,\widetilde{L}_{T_k}\in B_{\delta_0})\\
&=&\mathbb{E}^x\bigg[\mathbb{P}^{x}\bigg(\sup_{0\leq t\leq T_1}|\widetilde{L}_t|<\varepsilon,\widetilde{L}_{T_1}\in B_{\delta_0},\cdots,\sup_{T_{k-1}\leq t\leq T_k}|\widetilde{L}_t|<\varepsilon,\widetilde{L}_{T_k}\in B_{\delta_0}\bigg|\widetilde{L}_{T_{k-1}}=y_{k-1}\bigg)\bigg]\\
&=&\mathbb{P}^{x}\bigg(\sup_{0\leq t\leq T_1}|\widetilde{L}_t|<\varepsilon,\widetilde{L}_{T_1}\in B_{\delta_0},\cdots,\sup_{T_{k-2}\leq t\leq T_{k-1}}|\widetilde{L}_t|<\varepsilon,\widetilde{L}_{T_{k-1}}\in B_{\delta_0}\bigg)\\
&&\cdot\mathbb{P}^{y_{k-1}}(\sup_{0\leq t\leq T_1}|\widetilde{L}_t|<\varepsilon,\widetilde{L}_{T_1}\in B_{\delta_0})\\
&\cdots&\\
&=&\prod_{i=0}^{k-1}\mathbb{P}^{y_{i}}(\sup_{0\leq t\leq T_1}|\widetilde{L}_t|<\varepsilon,\widetilde{L}_{T_1}\in B_{\delta_0})\\
&\geq&(\frac{1}{4})^k.
\end{eqnarray*}
The result follows.
\end{proof}
Moreover, with respect to the L\'evy process defined in \eqref{levy-process}, we can obtain a more accurate order for the probability of the process staying in a ball during a period of time.
\begin{Lemma}\label{stable-property2}{\bf(Bilateral suprema)}
In the case of non-truncated, the L\'evy process defined in \eqref{levy-process} satisfies 
$$\mathbb{P}(\sup_{0\leq t\leq T}|L_t|<\varepsilon)\thicksim \exp(T\varepsilon^{-\alpha}),$$
for any $\varepsilon>0$, $T>0$.
In the cases of truncated,
$$\mathbb{P}(\sup_{0\leq t\leq T}|L_t|<\varepsilon)>c \exp(T\varepsilon^{-\alpha}),$$ 
with some positive constant $c$, for any $\varepsilon>0(<1)$, $T>0$. 
\end{Lemma}
\begin{proof}
By the characteristic of the jump process, for any $\varepsilon>0(\ll 1)$, the event 
$$\bigg\{\sup_{0\leq t\leq T}\bigg|\int_{B_1}z\widetilde{N}(t,\mathrm{d}z)+\int_{B_1^c}z N(t,\mathrm{d}z)\bigg|<\varepsilon\bigg\}$$ means $$\sup_{0\leq t\leq T}\bigg|\int_{B_1^c}z N(t,\mathrm{d}z)\bigg|=0.$$ 
This derives,
$$\mathbb{P}\bigg(\sup_{0\leq t\leq T}\bigg|\int_{B_1}z\widetilde{N}(t,\mathrm{d}z)\bigg|<\varepsilon\bigg)
> \mathbb{P}\bigg(\sup_{0\leq t\leq T}\bigg|\int_{B_1}z\widetilde{N}(t,\mathrm{d}z)+\int_{B_1^c}z N(t,\mathrm{d}z)\bigg|<\varepsilon\bigg).$$
Thus we just need to prove the result for the non-truncated case. This can be done by scaling and the related result of  Proposition 3 \cite{Jean Bertoin}.
\end{proof}

\begin{Lemma}\label{attractor}{\bf(Position of random equilibrium)}
In the cases of truncated and non-truncated with $\alpha\in(1,2)$, the random equilibrium $\{a_\beta(\omega)\}_{\omega\in\Omega}$ of system \eqref{eq0} can be stay in the vicinity of the origin with positive probability. That is, for any $\beta\in\mathbb{R}$, $\varepsilon>0$ and $T\geq0$, there exists a measurable set $\mathcal{A}\in\mathcal{F}_{-\infty}^T$ of positive measure such that
$$a_\beta(\theta_s\omega)\in(-\varepsilon,\varepsilon)\quad\mbox{for all }s\in[0,T]\quad\mbox{and }\omega\in \mathcal{A}.$$
\end{Lemma}

\begin{proof}
\begin{eqnarray}\label{lm2.8.1}
\rho(\cdot)=\int_\Omega\delta_{a(\omega)}(\cdot)\mathrm{d}\mathbb{P}(\omega)
\end{eqnarray}
Define 
$$\eta:=\frac{\varepsilon e^{-|\beta| T}}{3}.$$
The support of $\rho$ is the entire real line by Lemma \ref{levyproperty1}, thus combine with \eqref{lm2.8.1} the set 
\begin{eqnarray}
A_1:=\{\omega\in\Omega:a_\beta(\omega)\in(-\eta,\eta)\}
\end{eqnarray}
has positive probability for any $\beta\in\mathbb{R}$ and $A_1\in\mathcal{F}_{-\infty}^0$. Define 
$$A_2:=\{\omega\in\Omega:\sup_{t\in[0,T]}|L_t|\leq \eta e^{|\beta| T}\}\in\mathcal{F}_0^T.$$
Then $A_2$ has positive probability based on Lemma \ref{levyproperty2}. 
Set $\mathcal{A}:=A_1\cap A_2$. Then $\mathcal{A}\in\mathcal{F}_{-\infty}^T$ still has positive probability due to the independent of $A_1$ and $A_2$. Since $a_\beta(\omega)$ is a random equilibrium of $\varphi$, it follows that 
\begin{eqnarray}
a_\beta(\theta_t\omega)=a_\beta(\theta_s\omega)+\int_s^t(\beta a_\beta(\theta_r\omega)-a_\beta^3(\theta_r\omega))\mathrm{d}r+\sigma(L_t(\omega)-L_s(\omega)).
\end{eqnarray}
Let $\gamma(\theta_t\omega)=a_\beta(\theta_t\omega)-\sigma L_t(\omega)$. Then $\gamma(\theta_t\omega)$ is continuous. Refer to the proof of Proposition 4.1 \cite{Mark Callaway}, for $\omega\in\mathcal{A}$,
$$|\gamma(\theta_t\omega)|\leq \eta+\int_0^t|\beta||\gamma(\theta_s\omega)|\mathrm{d}s \quad\mbox{for all }t\in[0,T].$$
By Gronwall's inequality, it follows that 
$$|\gamma(\theta_t\omega)|\leq \eta e^{|\beta|t} \quad \mbox{for all }t\in[0,T].$$
Thus
$$|a_\beta(\theta_t\omega)|\leq \eta e^{|\beta|t}+\eta e^{|\beta| T}<\varepsilon \quad \mbox{for all }t\in[0,T].$$
\end{proof}
\section{Proof of the main results}\label{Proof of the main results}
{\bf Proof of Theorem \ref{Lya-expo}.} 
\begin{proof}
Lemma \ref{moment} assures that 
\begin{eqnarray*}
\int_{\mathbb{R}}|b'(x)|\rho(\mathrm{d}x)
=\int_{\mathbb{R}}|\beta-3x^2|\rho(\mathrm{d}x)
\leq|\beta|+3\int_{\mathbb{R}}x^2\rho(\mathrm{d}x)
\leq|\beta|+3c.
\end{eqnarray*}
By Remark \ref{Lyapunov}, the linear cocycle $\Phi$ satisfies the integrability conditions. According to the multiplicative ergodic theorem the Lyapunov exponent associated to the invariant measure $\mu=\delta_{a(\omega)}$ exists as a limit. Further it can be computed according to Birkhoff-Chintchin ergodic theorem \cite{Ludwig Arnold}, that is
\begin{eqnarray*}
\lambda&=&\lim_{t\to\infty}\frac{1}{t}\ln|\xi_t|
=\lim_{t\to\infty}\frac{1}{t}\int_0^t b'(a(\theta_{s}\omega))\mathrm{d}s
=\int_{\mathbb{R}}b'(x)\rho(\mathrm{d}x)
=\int_{\mathbb{R}}b'(x)p(x)\mathrm{d}x,
\end{eqnarray*}
where $\xi_t$ is the linearized flow along the invariant measure $\mu=\delta_{a(\omega)}$,
\begin{eqnarray}\label{linearization}
\mathrm{d}\xi_{t} = b'(a(\theta_{t}\omega))\xi_t\mathrm{d}t=[\beta-3a^2(\theta_{t}\omega)]\xi_t,
\end{eqnarray}
and $p(x)$ is the stationary density of the corresponding  Fokker-Planck equation. Note that the stationary density $p(x)$ satisfies 
%\begin{eqnarray}\label{Lya-expo-eq1}
%-\bigtriangledown\cdot(bp)-(-\bigtriangleup)^{\frac{\alpha}{2}}p=0.
%\end{eqnarray}
\begin{eqnarray}\label{Lya-expo-eq2}
\int_{\mathbb{R}} bf^{'}p\mathrm{d}x-\int_{\mathbb{R}}(-\bigtriangleup)^{\frac{\alpha}{2}}fp\mathrm{d}x=0.
\end{eqnarray}
Let $f=\ln p$ based on the strictly positive property of $p$ in Lemma \ref{levyproperty1}.  Then \eqref{Lya-expo-eq2} becomes to 
\begin{eqnarray}\label{Lya-expo-eq3}
\int_{\mathbb{R}} bp^{'}\mathrm{d}x-\int_{\mathbb{R}}(-\bigtriangleup)^{\frac{\alpha}{2}}\ln p \cdot p\mathrm{d}x=0.
\end{eqnarray}
%By \eqref{Lya-expo-eq1}, we can also get 
By the stationary Fokker-Planck equation, we can also get 
\begin{eqnarray}\label{Lya-expo-eq4}
\int_{\mathbb{R}} bp^{'}\mathrm{d}x+\int_{\mathbb{R}}(-\bigtriangleup)^{\frac{\alpha}{2}}p\mathrm{d}x+\int_{\mathbb{R}} b^{'}p\mathrm{d}x=0.
\end{eqnarray}
By \eqref{Lya-expo-eq3} and \eqref{Lya-expo-eq4}, we can determine the sign of the Lyapunov exponent,
\begin{eqnarray*}
\lambda&=&\int_{\mathbb{R}}b'p\mathrm{d}x\\
&=&-\int_{\mathbb{R}}(-\bigtriangleup)^{\frac{\alpha}{2}}p\mathrm{d}x-\int_{\mathbb{R}}(-\bigtriangleup)^{\frac{\alpha}{2}}\ln p \cdot p\mathrm{d}x\\
&=&-\int_{\mathbb{R}}(-\bigtriangleup)^{\frac{\alpha}{4}}\ln p \cdot (-\bigtriangleup)^{\frac{\alpha}{4}} p\mathrm{d}x\\
&<&0.
\end{eqnarray*} 
The last step holds by noticing the monotonicity of $\ln u$, and the definition of $(-\bigtriangleup)^{\frac{\alpha}{2}}$.
\end{proof}

%\begin{Definition}{\bf(Uniformly exponential attractive)}
{\bf Proof of Theorem \ref{uniformlyattractive}.}
\begin{proof}
(i) Take $x\in\R$ such that $x\neq a_{\beta}(\omega)$. Assume $x> a_{\beta}(\omega)$ without loss of generality. Then $\varphi(t,\omega,x)>\varphi(t,\omega,a_{\beta}(\omega))$ for all $t\geq 0$, according to the proof of Proposition \ref{Cllapse of random attractor}. Let $\gamma(t,\omega,x)=\varphi(t,\omega,x)-\sigma L_t$. Then $\gamma(t,\omega,x)$ is continuous, $\gamma(t,\omega,x)>\gamma(t,\omega,a_{\beta}(\omega))$ and 
\begin{eqnarray*}
\gamma(t,\omega,x)-\gamma(t,\omega,a_{\alpha}(\omega))
&=&\varphi(t,\omega,x)-\varphi(t,\omega,a_{\beta}(\omega))\\
&=&x-a_{\beta}(\omega)+\int_0^t[\beta\varphi(s,\omega,x)-(\varphi(s,\omega,x))^3]\mathrm{d}s\\
&&-\int_0^t[\beta\varphi(s,\omega,a_{\beta}(\omega))-(\varphi(s,\omega,a_{\beta}(\omega)))^3]\mathrm{d}s\\
&\leq& x-a_{\beta}(\omega)+\beta\int_0^t[\varphi(s,\omega,x)-\varphi(s,\omega,a_{\beta}(\omega))]\mathrm{d}s\\
&=&x-a_{\beta}(\omega)+\beta\int_0^t[\gamma(s,\omega,x)-\gamma(s,\omega,a_{\beta}(\omega))]\mathrm{d}s.
\end{eqnarray*}
In line with Lemma \ref{Gronwall}, 
$$|\gamma(t,\omega,x)-\gamma(t,\omega,a_{\beta}(\omega))|\leq e^{\beta t}|x-a_{\beta}(\omega)| \quad\mbox{for all }x\in\R.$$ 
(ii) The proof idea refer to Theorem $4.2$ \cite{Mark Callaway}.
Suppose to the contrary that there exists $\delta>0$ such that
$$\lim_{t\to0}\sup_{x\in(-\delta,\delta)}\esssup_{\omega\in\Omega}|\varphi(t,\omega,a_{\beta}(\omega)+x)-a_{\beta}(\theta_t\omega)|=0,$$
which implies that there exists $N\in\mathbb{N}$ such that
$$\sup_{x\in(-\delta,\delta)}\esssup_{\omega\in\Omega}|\varphi(t,\omega,a_{\beta}(\omega)+x)-a_{\beta}(\theta_t\omega)|<\frac{\sqrt{\beta}}{4}, \mbox{ for all }t\geq N.$$
For the above $\delta>0$,  by Lemma \ref{attractor} there exists $\mathcal{A}\in\mathcal{F}_{-\infty}^{0}$ of positive probability such that 
$$a_{\beta}(\omega)\in(-\frac{\delta}{2},\frac{\delta}{2}), \mbox{ for } \omega\in\mathcal{A}.$$
Note that $-\sqrt{\beta}$ and $\sqrt{\beta}$ are two attractive equilibria for the corresponding deterministic system with $\sigma=0$ in system \eqref{eq0}. Let $\phi(t,x_0)$ is the flow of the deterministic system started from $x_0$. Then there exists $T>N$ such that 
$$\phi(T,\frac{\delta}{2})>\frac{\sqrt{\beta}}{2} \quad\mbox{and }\quad \phi(T,-\frac{\delta}{2})<-\frac{\sqrt{\beta}}{2}.$$
For any $\varepsilon>0$, we define
$$\mathcal{A}_\varepsilon^+=\{\omega\in\Omega: \sup_{t\in[0,T]}|L_t|<\varepsilon\}.$$
Then $\mathcal{A}_\varepsilon^+\in\mathcal{F}_0^T$ and $\mathcal{A}_\varepsilon^+$ has positive probability by Lemma \ref{levyproperty2}. Thus $\mathbb{P}(\mathcal{A}\cap\mathcal{A}_\varepsilon^+)=\mathbb{P}(\mathcal{A})\mathbb{P}(\mathcal{A}_\varepsilon^+)$ is positive. According to Lemma \ref{disturbesde}, there exist $\varepsilon>0$ such that on the set $\mathcal{A}_\varepsilon^+$, we have
$$|\varphi(T,\omega,\frac{\delta}{2})-\phi(T,\frac{\delta}{2})|<\frac{\sqrt{\beta}}{4},\quad \mbox{ and } |\varphi(T,\omega,-\frac{\delta}{2})-\phi(T,-\frac{\delta}{2})|<\frac{\sqrt{\beta}}{4}.$$ 
While on the set $\mathcal{A}\cap\mathcal{A}_\varepsilon^+$, 
$$\sup_{x\in(-\delta,\delta)}|\varphi(T,\omega,a_\beta(\omega)+x)-a_\beta(\theta_T\omega)|\geq \max\bigg\{|\varphi(T,\omega,\frac{\delta}{2})-a_\beta(\theta_T\omega)|, \,\,|\varphi(T,\omega,-\frac{\delta}{2})-a_\beta(\theta_T\omega)|\bigg\}.$$
Consequently,
$$\sup_{x\in(-\delta,\delta)} \esssup_{\omega\in\Omega}|\varphi(T,\omega,a_\beta(\omega)+x)-a_\beta(\theta_T\omega)|>\frac{\sqrt{\beta}}{4},$$
which contradicts the assumption. The result follows.
\end{proof}

{\bf Proof of Theorem \ref{Fini-tim Lya exponent}.}
\begin{proof}
(i) follows directly from Theorem \ref{uniformlyattractive} (i).\\
(ii)We recall that $\Phi_\beta(t,\omega):=\frac{\partial\varphi_\beta}{\partial x}(t,\omega,a_\beta(\omega))$ is the linearized random dynamical system along the random equilibrium $a_\beta(\omega)$. The linearized equation along the random equilibrium $a_\beta(\omega)$ is given by
\begin{eqnarray}\label{linearized}
\dot{\xi_t}=(\beta-3a_\beta(\theta_t\omega)^2)\xi_t.
\end{eqnarray}
Thus 
\begin{eqnarray}\label{linearizedflow}
\Phi_\beta(t,\omega)=\exp\bigg(\int_0^t(\beta-3a_\beta(\theta_s\omega)^2))\mathrm{d}s\bigg).
\end{eqnarray}
%The result follows based on Proposition \ref{attractor} and similar to Theorem $4.3$ \cite{Mark Callaway}.
The finite-time Lyapunov exponent 
$$\lambda_\beta^{T,\omega}=\beta-\frac{1}{T}\int_0^t 3a_\beta(\theta_s\omega)^2\mathrm{d}s$$
Let $\varepsilon:=\frac{\sqrt{\beta}}{2}$. Then based on Proposition \ref{attractor} there exists a measurable set $\mathcal{A}\in\mathcal{F}_{-\infty}^T$ such that, 
$$\mathbb{P}\bigg(a_\beta(\theta_s\omega)^2\in(-\varepsilon,\varepsilon)\mbox{ for all }s\in[0,T]\bigg)>0.$$
The result follows since
$$\lambda_\beta^{T,\omega}\geq\frac{\beta}{4}, \mbox{ for }\,\omega\in\mathcal{A}.$$
\end{proof}

{\bf Proof of Theorem \ref{Dichotomy spectrum} on dichotomy spectrum} 
\begin{proof}
The explicit expression of the linearized flow is given by \eqref{linearizedflow} as follows
\begin{eqnarray*}
\Phi_\beta(t,\omega)=\exp\bigg(\int_0^t(\beta-3a_\beta(\theta_s\omega)^2))\mathrm{d}s\bigg).
\end{eqnarray*}
Hence, 
$$|\Phi_\beta(t,\omega)|\leq e^{\beta|t|} \quad\mbox{for all }t\in\mathbb{R},$$
and $\Sigma_\beta\subset[-\infty,\beta]$. In the following, we prove that $[-\infty,\beta]\subset \Sigma_\beta$. Otherwise, $\Phi_\beta$ admits an exponential dichotomy with growth rate for some $\gamma\in(-\infty,\beta]$ with invariant projector $P_\gamma$ and positive constants $K,\varepsilon$. Since  the system \eqref{eq0} is considered in $\mathbb{R}$, the invariant projector $P_\gamma$ can be two possible cases: (i) $P_\gamma=$id and (ii) $P_\gamma=0$.

{\bf Case (i).} $P_\gamma=$id. According to Proposition \ref{attractor}, for that $\varepsilon>0$ in the exponential dichotomy, there exists a measurable set $\mathcal{A}\in\mathcal{F}_{-\infty}^{T}$ of positive measure such that
$$a_\beta(\theta_s\omega)\in(-\frac{\sqrt{\varepsilon}}{2},\frac{\sqrt{\varepsilon}}{2})\quad\mbox{for all }\omega\in\mathcal{A}\quad\mbox{and }s\in[0,T].$$
We can get a contradiction from \eqref{linearization} with the exponential dichotomy for $\gamma$ \cite{Mark Callaway}.

{\bf Case (ii).} $P_\gamma=0$. On the basis of flow property and definition for exponential dichotomy, we have for almost all $\omega\in\Omega$
$$\Phi_\beta(t,\theta_{-t}\omega)^{-1}=\Phi_\beta(-t,\omega)\leq Ke^{(\gamma+\varepsilon)(-t)}\quad\mbox{for all }t\geq0.$$
Namely, almost surely, 
$$\Phi_\beta(t,\theta_{-t}\omega)\geq \frac{1}{K}e^{(\gamma+\varepsilon)t}\quad\mbox{for all }t\geq0.$$
Define 
$$\mathcal{A_\varepsilon^-}=\{|a_{\beta}(\omega)|\in(-\varepsilon,-\varepsilon)\},$$ 
and 
$$\mathcal{A_\varepsilon^+}=\{\sup_{t\in[0,T]}|L_t-\frac{t^4}{4}+\beta\frac{t^2}{2}-t|\leq\varepsilon\}.$$
Then for any $\varepsilon>0$, $\mathcal{A_\varepsilon^-}\in\mathcal{F}_{-\infty}^0$ and $\mathcal{A_\varepsilon^+}\in\mathcal{F}_0^T$ are independent positive probability sets assured by Lemma \ref{attractor} and Lemma \ref{levyproperty2}. Consider 
$$x_t=\int_0^t(\beta x_s-x_s^3)\mathrm{d}s+\frac{t^4}{4}-\beta\frac{t^2}{2}+t.$$
According to Lemma \ref{disturbesde}, there exists $\varepsilon>0$, such that for all $\omega\in\mathcal{A_\varepsilon^-}\cap\mathcal{A_\varepsilon^+}$, 
$$\sup_{t\in[0,T]}|a_{\beta}(\theta_t(\omega))-t|\leq1.$$
We will arrive at a contradiction \cite{Mark Callaway} with 
$$\frac{\ln K + (\beta-\gamma)T}{3}<\int_0^T a_\beta(\theta_s\omega)\mathrm{d}s\leq\frac{\ln K + (\beta-\gamma)T}{3}$$
for $\omega\in\theta_T(\mathcal{A_\varepsilon^-}\cap\mathcal{A_\varepsilon^+})$ and some $T>0$ sufficiently large, where $\mathcal{A_\varepsilon^-}\cap\mathcal{A_\varepsilon^+}\subset\Omega$ is a positive probability set.
In conclusion, $$\Sigma_\beta=[-\infty,\beta] \quad \mbox{for all }\beta\in\mathbb{R}.$$
\end{proof}

\appendix
\section{Appendix}
\begin{Lemma}{\bf(Generalized Gronwall inequality)}\label{Gronwall}
Let $t\to f(t,\omega)$ be a continuous function almost surely in $\mathbb{R}$ satisfing
$$f(t)\leq f(s)+\beta\int_s^tf(u)\mathrm{d}u, \quad a.s. \quad 0\leq s<t<\infty,$$
where $\beta<0$. Then
$$f(t)\leq f(0)e^{\beta t}\quad a.s. \quad 0\leq t<\infty.$$
\end{Lemma}
\begin{proof}
Let $f_0(t)=f(0)e^{\beta t}$ and $f_1(t)=f(t)-f_0(t)$. Then $f_1(0)=0$, function $f_1(t)$ is continuous almost surely and $f'_0(t)=\beta f_0(t)$. Hence
$$f_0(t)-f_0(s)=\beta\int_s^t f_0(u)\mathrm{d}u.$$
We obtain
\begin{eqnarray*}
f_1(t)-f_1(s)&=&(f(t)-f_0(t))-(f(s)-f_0(s))\\
&\leq&\beta\int_s^t f(u)\mathrm{d}u-\beta\int_s^t f_0(u)\mathrm{d}u\\
&=&\beta\int_s^t f_1(u)\mathrm{d}u.
\end{eqnarray*}
That is $f_1(t)-f_1(s)\leq \beta\int_s^t f_1(u)\mathrm{d}u$ and $f_1(0)=0$. The result follows based on the proof of Lemma $8.1$ \cite{Kiyosi Ito and Makiko Nisio}.
\end{proof}
\begin{Lemma}{\bf (comparison principle)}\label{disturbesde}
In the cases of truncated and non-truncated, system \eqref{eq0} and the deterministic system 
$$\dot{x}_t=\beta x_t-x_t^3+\sigma g(t),$$
with continuous function $g(t)$.
For any $T>0$ and $\delta>0$, there exists $\varepsilon>0$ and a positive probability set $\mathcal{A}^{\varepsilon}\in\mathcal{F}_0^T$, such that for $\omega\in\mathcal{A}$ and $|X_0-x_0|<\varepsilon$,
$$\sup_{t\in[0,T]}|X_t-x_t|\leq \delta.$$
\end{Lemma}
\begin{proof}
By Lemma \ref{levyproperty2}, the set $\mathcal{A}^{\varepsilon}=\{\underset{0\leq t\leq T}{\sup}|L_t-g(t)|<\varepsilon\}\in\mathcal{F}_0^T$ has positive probability for $\varepsilon=\frac{\delta e^{-|\beta|T}}{1+\sigma}$. Hence for $\omega\in\mathcal{A}^{\varepsilon}$ and $t\in[0,T]$,
\begin{eqnarray*}
|X_t-x_t|&\leq& |X_0-x_0|+\int_0^t |\beta||X_s-x_s|\mathrm{d}s+\sigma|L_t-g(t)|\\
&\leq&(1+\sigma)\varepsilon+\int_0^t |\beta||X_s-x_s|\mathrm{d}s.
\end{eqnarray*}
Gronwall inequality gives that,
\begin{eqnarray*}
|X_t-x_t|&\leq&\delta.
\end{eqnarray*}
\end{proof}


\begin{thebibliography}{99}
\bibliographystyle{plain}

 \bibitem{David Applebaum}
D. Applebaum. 
\newblock L\'evy Processes and Stochastic Calculus.
\newblock  Cambridge University Press, 2009.


 \bibitem{Ludwig Arnold}
L. Arnold.
\newblock Random Dynamical Systems.
\newblock  New York: Springer-Verlag, 1998.

\bibitem{Jean Bertoin}
J. Bertoin.
\newblock L\'evy Processes.
\newblock  Cambridge University Press, 1996.

%\bibitem{V.I. Bogachev}
%V.I. Bogachev, N.V. Krylov and M. R$\ddot{o}$ckner.
%\newblock On regularity of transition probabilities and invariant measures of singular diffusions under minimal conditions.
%\newblock  Communications in Partial Differential Equations, 26(11\&12), 2037-2080, 2001.

%\bibitem{Krzysztof Bogdan}
%K. Bogdan and T. Grzywny.
%\newblock Heat kernel of fractional Laplacian in cones.
%\newblock  Colloquium Mathematicum, 2(118), 365-377, 2010.

\bibitem{Claudia Bucur}
C. Bucur and E. Valdinoci.
\newblock Nonlocal Diffusion and Applications.
\newblock  Springer, 2016.

\bibitem{Mark Callaway}
M. Callaway, T. Doan, J. Lamb and M. Rasmussen.
\newblock The dichotomy spectrum for random dynamical systems and pitchfork bifurcations with additive noise.
\newblock  Annales de I'Institut Henri Poinc$\acute{a}$re-Probabilit\'es et Statistiques, 4(53), 1548-1574, 2017.


\bibitem{Sandra Cerrai}
S. Cerrai.
\newblock Averaging principle for systems of Reaction-Diffusion equations with polynomial nonlinearities perturbed by multiplicative noise.
\newblock  SIAM Journal on Mathematical Analysis, 43(6), 2482-2518, 2011.


%\bibitem{Zhenqing Chen}
%Z. Chen, P. Kim and R. Song.
%\newblock Dirichlet heat kernel estimates for rotationally symmetric L\'evy processes.
%\newblock  Proceedings of the London Mathematical Society, 109(2014), 90-120, 2014.

\bibitem{Hans Crauel and Franco Flandoli}
H. Crauel and F. Flandoli.
\newblock Additive noise destroys a pitchfork bifurcation.
\newblock  Journal of Dynamics and Differential Equations, 10(2), 259-274, 1998.

\bibitem{Hans Crauel Markov}
H. Crauel.
\newblock Markov measures for random dynamical systems.
\newblock  Stochastics and Stochastic Reports, 37, 153-173, 1991.

\bibitem{JinqiaoDuan}
J. Duan.
\newblock  An Introduction to Stochastic Dynamics.
\newblock Cambridge University Press, 2015.

\bibitem{Martin Friesen Peng Jin Jonas Kremer and Barbara Rudiger}
M. Friesen, P. Jin, J. Kremer and B. R$\ddot{\mbox{u}}$diger.
\newblock Exponential ergodicity for stochastic equations of nonnegative processes with jumps.
\newblock  arXiv preprint arXiv: 1902.02833, 2019.

\bibitem{Michael Rockner}
B. Gess, W. Liu and M. R$\ddot{\mbox{o}}$ckner.
\newblock Random attractors for a class of stochastic partial differential equations driven by general additive noise.
\newblock  Journal of Differential Equations, 251(2011), 1225-1253, 2011.

\bibitem{Martin Hairer and Jonathan C. Mattingly}
M. Hairer and J. C. Mattingly.
\newblock Yet another look at Harris’ ergodic theorem for Markov chains.
\newblock  Seminar on Stochastic Analysis, Random Fields and Applications VI. Springer, Basel, 109-117, 2011.

\bibitem{Qing Han}
Q. Han and F.H. Lin.
\newblock Elliptic Partial Differential Equations.
\newblock  American Mathematical Society, 2000.

%\bibitem{Qiao Huang}
%Q. Huang, J.Q. Duan and J.L. Wu.
%\newblock Maximum principles for nonlocal parabolic Waldenfels operators.
%\newblock  Bulletin of Mathematical Sciences, 2018.


\bibitem{Kiyosi Ito and Makiko Nisio}
K. It$\hat{\mbox{o}}$ and M. Nisio.
\newblock On stationary solutions of a stochastic differential equation.
\newblock  Journal of Mathematics of Kyoto University, 4(1), 1-75, 1964.

\bibitem{Ioannis Karatzas and Steven E. Shreve}
I. Karatzas and Steven E. Shreve.
\newblock Brownian Motion and Stochastic Calculus.
\newblock  Springer, 2000.

%\bibitem{Alexey Kulik}
%A. Kulik. 
%\newblock Exponential ergodicity of the solutions to SDE's with a jump noise.
%\newblock  Stochastic processes and their applications, 119(2009), 602-632, 2008.

\bibitem{Mateusz Kwasnicki}
M. Kwa$\acute{\mbox{s}}$nicki. 
\newblock Ten equivalent definitions of the fractional Laplace operator.
\newblock  Fractional Calculus $\&$ Applied Analysis, 20(1), 7-51, 2017.

\bibitem{Claudia Prevot and Michael Rockner}
C. Pr\'ev$\hat{\mbox{o}}$t and M. R$\ddot{\mbox{o}}$ckner.
\newblock A Concise Course on Stochastic Partial Differential Equations, Lecture Notes in Mathematics.
\newblock  Springer, 1905.

\bibitem{Ken-Iti Sato}
K.I. Sato.
\newblock L\'evy Processes and Infinitely Divisible Distributions.
\newblock  Cambridge University Press, 1999.

\bibitem{Michael Sharpe}
M. Sharpe.
\newblock Zeroes of infinitely divisible densities.
\newblock  The Annals of Mathematical Statistics, 40(4), 1503-1505, 1969.

\bibitem{S.J.Taylor}
S.J. Taylor.
\newblock Sample path properties of a transient stable process.
\newblock  Journal of Mathematics and Mechanics, 16(11), 1229-1246, 1967.

\bibitem{Jian Wang}
J. Wang.
\newblock $L^p$-Wasserstein distance for stochastic differential equations driven by L\'evy Processes.
\newblock  Bernoulli, 22(3), 1598-1616, 2016.

% \bibitem{Karatzas}
%Karatzas L. and Shreve S.:
%\newblock Brownian motion and stochastic calculus.
%\newblock  1991, Springer.

%\bibitem{Schmalfuss}
%Schmalfuss B.:
%\newblock  Lyapunov functions and non-trivial stationary solutions of stochastic differential equations.
%\newblock Dynamical systems, 16(4), 303-317, 2001.

\end{thebibliography}
\end{document}